\documentclass[12pt,a4paper]{article}
\usepackage{amsmath,amssymb,amsfonts}
\usepackage[latin1]{inputenc}

\def\Bbb{\mathbb}

\title{\bf Digit frequencies and class numbers\thanks{MSC2020: 11A63; 11R29; 11R18. Keywords: Digit expansions; class numbers; generalized Bernoulli numbers.}}

\author{Kurt Girstmair}

\date{}

\makeatletter
\let\@@maketitle=\maketitle
\def\maketitle{\def\thispagestyle##1{\relax}\@@maketitle}
\makeatother
\newtheorem{theorem}{Theorem}

\def\BE{\begin{equation}}
\def\EE{\end{equation}}
\def\BD{\begin{displaymath}}
\def\ED{\end{displaymath}}
\def\BA{\begin{array}}
\def\EA{\end{array}}
\def\BEA{\begin{eqnarray*}}
\def\EEA{\end{eqnarray*}}
\def\BI{\bibitem}

\def\Z{\Bbb Z}
\def\Q{\Bbb Q}

\def\phi{\varphi}
\def\EPS{\varepsilon}

\def\MB{\mbox}
\def\LD{\ldots}
\def\OV{\overline}

\def\WH{\widehat}

\def\ORD{\mbox{ord}}

\def\DIV{\,|\,}
\def\NDIV{\, \nmid \,}

\def\MN{\medskip\noindent}
\def\STOP{\hfill$\Box$}

\def\BCH{B_{\chi}}
\def\BCHPS{B_{\chi\psi}}
\def\BCHPSF{B_{\chi\psi_5}}

\def\LS#1#2{ \left( \frac{#1}{#2} \right) }
\def\LF#1#2{\left\lfloor\frac{#1}{#2}\right\rfloor}

\def\RE{{\rm Re}}

\begin{document}
\maketitle
\normalsize

\begin{abstract}

Let $p>3$ be a prime and $b\ge 2$ an integer such that $p$ does not divide $b$. Then $1/p$ has a periodic digit expansion with respect to the basis $b$. The length $m$ of the period
is the (multiplicative) order of $b$ mod $p$. In the case $m=(p-1)/2$,
the frequency of each digit $k\in\{0,1,\LD b-1\}$ can be expressed in terms of generalized (first order) Bernoulli numbers. In some cases only Bernoulli numbers belonging to quadratic characters occur.
This means that the frequencies can be written in terms of class numbers of
imaginary quadratic number fields (the so-called class number case). In the present paper we classify the numbers $p$ and $b$ falling under the class number case.
We also highlight one of the simplest examples not falling under this
case ($p\equiv 1$ mod $4$, $b=10$) and the most complex example of the class number case. Moreover, we show that knowing the frequency of each $k$ is equivalent to knowing
the respective Bernoulli numbers.

\end{abstract}

%%%%%%%%%%%%%%%%%%%%%%%%%%%%%%%%%%%%%%%%%%%%%
\section*{1. Introduction}
%%%%%%%%%%%%%%%%%%%%%%%%%%%%%%%%%%%%%%%%%%%%%

Let $p>3$ be a prime, $b\ge 2$ an integer such that $p\NDIV b$.
Then
\BE
\label{1.1}
  \frac 1p=\sum_{j=1}^{\infty}c_jb^{-j},
\EE
where the numbers $c_j\in\{0,1,\LD,b-1\}$ are the {\em digits} of $1/p$ with respect to the basis $b$. It is well-known that the sequence of the digits is periodic and
that $(c_1,\LD,c_m)$ is a period, $m$ being the (multiplicative) order of $b$ mod $p$; see \cite{Gi1}.

Let $m=(p-1)/2$ be the order of $b$ mod $p$. Hence $(c_1,\LD,c_m)$ is a period of the digit expansion (\ref{1.1}) of $1/p$. For $k\in\{0,1,\LD,b-1\}$, the number
\BD
  n_{b,k}=|\{j\in\{1,\LD,m\}; c_j=k\}
\ED
is the {\em frequency} of the digit $k$ in the period $(c_1,\LD,c_m)$.
This paper is devoted to the study of the frequencies $n_{b,k}$.

Previously, we described these frequencies in the case $p\equiv 3$ mod $4$, $b=10$, in terms of the class numbers of two imaginary quadratic number fields, namely, $\Q(\sqrt{-p})$ and $\Q(\sqrt{-5p})$.
We also did this for $b=8$ and arbitrary primes $p>3$; see \cite{Gi2}. Our main tools were results of \cite{Be}. Actually, the paper \cite{Be} provides the tools for describing the frequencies in terms
of class numbers in several other cases.
For instance, if $b=24$, the said paper yields the frequencies $n_{24,k}$ for several values $k\in\{0,1,\LD,23\}$ in terms of the class numbers of up to four imaginary quadratic number fields; see Section 4.

For this reason the following question is obvious: Under which conditions can we describe the digit frequencies in terms of class numbers of imaginary quadratic fields? We answer this question in Sections 2 and 3,
on carving out what we call the {\em class number case}; see Theorem \ref{t1}.

If $p$ and $b$ do not fall under this case, the digit frequencies can be described by generalized Bernoulli numbers (of order 1) belonging to Dirichlet characters of an order $\ge 3$.
Section 4 is devoted to the case $p\equiv 1$ mod $4$, $b=10$, and to the case $b=24$.
Whereas $p\equiv 1$ mod $4$, $b=10$, is one the simplest examples not falling under the class number case, $b=24$ represents the most complex class number case.

In Section 5 we briefly show that the respective Bernoulli numbers
can be reconstructed, in a uniform way, from the frequencies $n_{b,k}$, $k=0,\LD,b-1$. In other words, the information
encoded in the digit frequencies is basically equivalent to the information encoded in the respective Bernoulli numbers.

Connections between the digits of rational numbers and class number factors occur in a number of papers; see \cite{Gi0}, \cite{Hi}, \cite{MuTh},
\cite{ChKr},  \cite{Mi}, \cite{PuSa}, and \cite{Sh}.

%%%%%%%%%%%%%%%%%%%%%%%%%%%%%%%%%%%%%%%%%%%%%%%%%%%%%
\section*{2. Digit frequencies and Bernoulli numbers}
%%%%%%%%%%%%%%%%%%%%%%%%%%%%%%%%%%%%%%%%%%%%%%%%%%%%%

Let $p$ and $b$ be as above.
For an integer $k$ we define $(k)_p$ as the integer $j\in\{0,1,\LD,p-1\}$ that satisfies
$j\equiv k$ mod $p$. We put
\BD
  \theta_b(k)=(b(k)_p-(bk)_p)/p
\ED
for $k\in\Z$. Then the digit $c_j$ of $1/p$ has the form
\BD
  c_j=\theta_b(b^{j-1}),
\ED
for all $j\ge 1$; see \cite{Gi1}.
For $k\in\{0,1,\LD b-1\}$ we have
\BE
\label{1.3}
 c_j=k \MB{ if, and only if, } (b^{j-1})_p\in \left(\frac{kp}b,\frac{(k+1)p}b\right);
\EE
see \cite{Gi2}. Let $Q$ be the set of quadratic residues  mod $p$ in $\{1,\LD,p-1\}$
and $N$ the set of quadratic non-residues mod $p$ in $\{1,\LD,p-1\}$. Since the numbers $(b^{j-1})_p$, $j=1,\LD,q$, run through $Q$, we obtain
\BD
 n_{b,k}= \left|Q\cap\left(\frac{kp}b,\frac{(k+1)p}b\right)\right|.
\ED
Let
\BE
\label{1.6}
  \delta_{b,k}=\left|Q\cap\left(\frac{kp}b,\frac{(k+1)p}b\right)\right|-\left|N\cap\left(\frac{kp}b,\frac{(k+1)p}b\right)\right|.
\EE
Then we have
\BE
\label{1.7}
  n_{b,k}=\frac 12 \left(\left|\Z\cap\left(\frac{kp}b,\frac{(k+1)p}b\right)\right|+\delta_{b,k}\right)
\EE
for $k=0,1,\LD,b-1$. Hence the frequencies $n_{b,k}$ are known if the numbers $\delta_{b,k}$ are known. These numbers, however, can be expressed in terms of generalized Bernoulli numbers.

Indeed, for $k\in\{0,1,\LD,b\}$ let
\BD
  \gamma_{b,k}=\sum_{1\le r\le kp/b}\chi(r),
\ED
where $\chi$ is the quadratic Dirichlet character defined by the Legendre symbol, i.e.,
\BD
 \chi(r)=\LS rp.
\ED
In particular, $\gamma_{b,0}=0=\gamma_{b,b}$.
By (\ref{1.6}), we have
\BE
\label{1.8}
  \delta_{b,k}=\gamma_{b,k+1}-\gamma_{b,k},
\EE
$k=0,\LD,b-1$. Let $d=(b,k)$. Then
\BE
\label{1.10}
  \gamma_{b,k}=\gamma_{b/d,k/d}.
\EE
If $(b,k)=1$, we have the formula
\BE
\label{1.12}
 \gamma_{b,k}=-\BCH+\frac{\chi(b)}{\phi(b)}\sum_{\psi\in X_b}\OV{\psi}(-kp)\BCHPS;
\EE
see \cite[Eq. (6)]{Sz}.
Here $\phi$ denotes Euler's function, $X_b$ the set of all Dirichlet characters mod $b$, and the bar means complex conjugation.
The Bernoulli numbers $\BCH$ and $\BCHPS$ are defined by
\BD
   \BCH=\frac 1p\sum_{j=1}^{p-1}j\chi(j) \MB{ and }\BCHPS=\frac 1{bp}\sum_{j=1}^{bp-1}j\chi(j)\psi(j).
\ED
If we combine (\ref{1.8}), (\ref{1.10}),  and (\ref{1.12}), we can express $\delta_{b,k}$ in terms of the generalized Bernoulli numbers $\BCH$ and $\BCHPS$, where $\psi$ runs through
the sets $X_d$ for the divisors $d$ of $b$.

Now we can describe the {\em class number case} more precisely. In fact, this case occurs if, and only if, all non-vanishing Bernoulli numbers $\BCHPS$, $\psi\in X_b$, in (\ref{1.12}) belong to quadratic characters.
Accordingly, the numbers $-\BCHPS$ are, for these characters, class numbers of imaginary quadratic number fields. It turns out that this definition is strong enough. Indeed, it restricts the possible numbers $b$
to few cases, in which also the non-vanishing numbers $\BCHPS$, $\psi\in X_d$, $d$ a divisor of $b$, belong to characters $\psi$ of an order $\le 2$.

We have the following result.

\begin{theorem} % Theorem 1 %%%%%%%%%%%%%%%%%%%%%%%%%%%%%%%%%%%%%%%
\label{t1}

If $p\equiv 3$ mod 4, then the class number case occurs only for $b\in\{5,10\}$ and for divisors  $b$ of $24$, $b\ge 2$.
If $p\equiv 1$ mod 4, then the class number case occurs only for divisors  $b$ of $24$, $b\ge 2$.

\end{theorem} %%%%%%%%%%%%%%%%%%%%%%%%%%%%%%%%%%%%%%%%%%%%%%%%%%%%%

Suppose that $p>3$ and $b$ do not fall under the class number case.
In particular, let $\psi\in X_b$ be such that $\ORD(\psi)\ge 3$ and $\BCHPS\ne 0$. Then $\OV{\psi}$ also belongs to $X_b$,
and $B_{\chi\OV{\psi}}\ne 0$. Accordingly, a rational multiple of
\BD
   2{\rm Re}(\OV{\psi}(-kp)\BCHPS)=\OV{\psi}(-kp)\BCHPS+\psi(-kp)B_{\chi\OV{\psi}}
\ED
occurs in $\gamma_{b,k}$ (see formula (\ref{1.12})), and, in particular, in $\delta_{b,0}$ and $n_{b,0}$.

%%%%%%%%%%%%%%%%%%%%%%%%%%%%%%%%%%%%%%%%%%%%%
\section*{3. The class number case}
%%%%%%%%%%%%%%%%%%%%%%%%%%%%%%%%%%%%%%%%%%%%%

If $p\equiv 3$ mod 4, then $\chi$ is an odd character, and, accordingly, (\ref{1.12}) can be written
\BE
\label{2.4}
 \gamma_{b,k}=-B_{\chi}+\frac{\chi(b)}{\phi(b)}\sum_{\psi\in X_b^+}\OV{\psi}(kp)\BCHPS,
\EE
where $X_b^+$ is the set of all {\em even} Dirichlet characters mod $b$.
In the case $p\equiv 1$ mod 4, $\chi$ is an even character, and, thus,
\BE
\label{2.6}
 \gamma_{b,k}=-\frac{\chi(b)}{\phi(b)}\sum_{\psi\in X_b^-}\OV{\psi}(kp)\BCHPS,
\EE
where $X_b^-$ is the set of all {\em odd} Dirichlet characters mod $b$.

Let $\psi\in X_b$ and
$f$ be the conductor of $\psi$. Suppose that $\chi\psi$ is odd. Let $\psi_f$ be the character mod $f$ that induces $\psi$. Then $B_{\chi\psi_f}\ne 0$.
Moreover,
\BE
\label{2.7}
   \BCHPS=\prod_{q\DIV b}(1-\chi\psi_f(q))B_{\chi\psi_f},
\EE
$q$ running through all primes dividing $b$; see \cite[p. 275]{Sz}. In particular, $\BCHPS\ne 0$ if all prime divisors of $b$ divide $f$.

{\em Proof of Theorem \ref{t1}.} We write $b=b_1b_2\cdots b_r$, where the numbers $b_i=q_i^{e_i}$ are powers of the primes $q_i$, and $q_1>q_2>\cdots>q_r$.
In view of (\ref{2.4}) and (\ref{2.6}), we try to construct, in the case $p\equiv 3$ mod $4$, an even character $\psi$ mod $b$ such that
$\BCHPS\ne 0$ and $\ORD(\psi)\ge 3$, and in the case $p\equiv 1$ mod $4$, an odd character $\psi$ of this kind. If our construction is successful,
then $b$ does not fall under the class number case.

Suppose, first, that $b\not\equiv 2$ mod $4$. If $b$ is odd, then there are quadratic characters $\psi_i$ mod $q_i$, $i=2,\LD, r$. If $b$ is even, then there are quadratic characters $\psi_i$ mod $q_i$,
$i=2,\LD,r-1$ and a quadratic character $\psi_r$ mod $4$. Put $\EPS=\psi_2\cdots \psi_r(-1)$.
We start with the case $q_1\ge 7$. Then there is an odd character $\psi_1$ mod $b_1$  with $\ORD(\psi_1)$ divisible by $(q_1-1)/2$ as well as an even character of
this kind. If we choose $\psi_1$ suitably and put $\psi=\psi_1\cdots\psi_r$, then $\psi(-1)=\pm\EPS$ takes the required value.
The character $\psi$ lies in $X_b$, since each $\psi_i$ is also a character mod $b_i$, $i=1,\LD,r$. Its  order is divisible by $(q_1-1)/2\ge 3$.
Since the conductor of $\psi$ is divisible by $q_1,\LD,q_r$, we have
$\BCHPS\ne 0$, as we observed above.

If $b\equiv 2$ mod $4$ the same construction gives a character $\psi$ mod $b$ such that $\ORD(\psi)\ge 3$ and its conductor $f$ is divisible by $q_1\cdots q_{r-1}$.
Let $\psi_f$ be the character mod $f$ that induces $\psi$. By (\ref{2.7}), we have
\BD
  \BCHPS=(1-\chi(2)\psi_f(2))B_{\chi\psi_f}.
\ED
In order to exclude $\psi_f(2)=\pm 1$, it suffices to exclude $\psi_1(2)=\pm 1$. Let $g$ be a primitive root mod $q_1$. Then $2\equiv g^k$ for some integer $k$ and $\psi_1(2)=\psi_1(g)^k$.
But $\psi_1$ has been chosen such that $\ORD(\psi_1)\equiv 0$ mod $(q_1-1)/2$. Therefore, $\psi_1(g)$ has an order divisible by $(q_1-1)/2$. Suppose that $\psi_1(2)=\pm 1$. Then $\psi_1(g)^{2k}=1$,
so $(q_1-1)/2\DIV 2k$ and $q_1-1\DIV 4k$. Accordingly, $g^{4k}\equiv 2^4\equiv 1$ mod $q_1$ and, thus, $q_1\DIV 15$, a contradiction.

We may suppose $q_1\le 5$ now. Let $q_1=5$ and $e_1\ge 2$. Then we can choose an odd character $\psi_1$ mod $b_1$ with $\ORD(\psi_1)\equiv 0$ mod $10$ and  conductor $5^2$, as well as an even character of this kind.
 In the case $b\not\equiv 2$ mod $4$, we find the character $\psi=\psi_1\cdots\psi_r$ in the above way. In the case $b\equiv 2$ mod $4$, the conductor of $\psi$ is divisible by $5^2q_2\cdots q_{r-1}$.
This means that we must exclude $\psi_1(2)=\pm 1$. But $2$ is a primitive root mod $25$, hence the order of $\psi_1(2)$ is $\equiv 0$ mod 10. Then $\psi_1(2)^2=1$ is impossible.

So we have to deal with $q_1=b_1=5$. There is a character $\psi_1\in X_5^-$ of order $4$. If $p\equiv 1$ mod $4$, we choose $\psi$ as the character mod $b$ induced by $\psi_1$. Then
\BE
\label{2.8}
  \BCHPS=(1-\chi\psi_1(2))^k(1-\chi\psi_1(3))^l B_{\chi\psi_1},\enspace k,l\in\{0,1\}.
\EE
However, the numbers $2$ and $3$ are both primitive roots mod $5$, which means $\psi_1(2),\psi_1(3)\in\{\pm i\}$. By (\ref{2.8}), $\BCHPS$ cannot vanish.

Let $p\equiv 3$ mod $4$. In the case $r=3$, we choose the quadratic character $\psi_2\in X_3^-$. Let $\psi$ be the (even) character mod $b$ induced by $\psi_1\psi_2$. Then
\BD
\label{2.10}
  \BCHPS=(1-\chi\psi_1\psi_2(2))B_{\chi\psi_1\psi_2}.
\ED
Since  $\psi_1(2)=\pm i$, $\BCHPS$ cannot vanish. In the case $r=2$ and $q_2=3$, we choose $\psi_2\in X_3^-$ as in the foregoing case and put $\psi=\psi_1\psi_2$.
In the case $r=2$, $q_2=2$ and $e_2\ge 2$, we choose $\psi_2\in X_4^-$ and argue in the same way. Hence only the cases $b=5$ and $b=10$ are left over.

If $q_1\in\{2,3\}$, we have to exclude the cases $9\DIV b$ and $16\DIV b$. This can be done by means of similar considerations. For instance, suppose that $r=2$ and $e_1\ge 2$. Then $e_2\ge 1$.
We have a character in $X_9^-$ of order $6$ and a character in $X_9^+$ of order $3$. Let $\psi_1$ be one of these characters, its sign being as required by the conditions $p\equiv 3$ mod $4$ and $p\equiv 1$ mod $4$.
Let $\psi$ be the character mod $b$ induced by $\psi_1$. Then
\BD
 \BCHPS=(1-\chi\psi_1(2))B_{\chi\psi_1}.
\ED
But $2$ is a primitive root mod $9$, and, accordingly, $\psi_1(2)$ a primitive $6$th or $3$rd root of unity, which means $\BCHPS\ne 0$.

In order to exclude $16\DIV b$, one observes that both $X_{16}^+$ and $X_{16}^-$ contain a character of order $4$.
\STOP

Theorem \ref{t1} says when the class number case is possible. We still have to show that it actually occurs for $p\equiv 3$ mod $4$, $b\in\{5,10\}$, and for arbitrary $p>3$ and divisors $b$ of $24$, $b\ge 2$.

The set $X_5^+$ consists of the principal character and $\psi_5^+$, the latter being defined by the Legendre symbol, i.e., $\psi_5^+(j)=\LS j5$. The set $X_{10}^+$ consists of the characters mod $10$ induced
by these two characters. Note that the only proper divisors $d\ge 2$ of $10$ are $2$ and $5$; $X_2^+$ consists of the principal character and $X_5^+$ has just been described.

For divisors $b\ge 2$ of 24 all characters mod $b$ have an order $\le 2$. This is clear for $b\in\{2,3, 4,6\}$, since the order of the group $X_b$ is $\le 2$, and for $b=8$,
since the group $X_8$ is isomorphic to $(\Z/8\Z)^{\times}$,
a group containing only elements of an order $\le 2$. The canonical decomposition $\psi=\psi_3\psi_d$ of a character $\psi\in X_b$, $b\in\{12, 24\}$, into characters mod $3$ and mod $d$, $d\DIV 8$,
shows that $\psi$ has also an order $\le 2$.

%%%%%%%%%%%%%%%%%%%%%%%%%%%%%%%%%%%%%%%%%%%%%%%%%%%%%%%%%%%%%%%%%%%%%%%
\section*{4. The case $p\equiv 1$ mod $4$, $b=10$, and the case $b=24$}
%%%%%%%%%%%%%%%%%%%%%%%%%%%%%%%%%%%%%%%%%%%%%%%%%%%%%%%%%%%%%%%%%%%%%%%

We adopt the above notation.
The case $p\equiv 1$ mod $4$, $b=10$, is one of the simplest examples not falling under the class number case. Indeed, it involves only the Bernoulli numbers $\BCHPSF$, $B_{\chi\OV{\psi_5}}$, where $\chi(j)=\LS jp$ and
the odd character $\psi_5$ mod $5$ is defined by $\psi_5(2)=i$.

In the case $p\equiv 1$ mod $4$, $b\ge 2$, we have the symmetry
\BD
 n_{b, b-1-k}=n_{b,k},
\ED
$k=0,\LD,b-1$. In fact, for $j\in\{1,\LD,p-1\}$,
\BD
 j\in \left(\frac{kp}b,\frac{(k+1)p}b\right) \MB{ if, and only if, } p-j\in \left(\frac{(b-k-1)p}b,\frac{(b-k)p}b\right),
\ED
and $j\in Q$ if, and only if $p-j\in Q$, since $\LS{-1}p=1$. The said symmetry follows from (\ref{1.6}) and (\ref{1.7}).

This means that we can restrict ourselves to $n_{10,0},\LD,n_{10, 4}$ in the case $p\equiv 1$ mod $4$, $b=10$ in question.
By (\ref{1.8}) and (\ref{1.10}), we obtain the decisive quantities $\delta_{10,0},\LD,\delta_{10,4}$
from $\gamma_{10,1}$, $\gamma_{10,2}=\gamma_{5,1}$, $\gamma_{10,3}$, $\gamma_{10,4}=\gamma_{5,2}$, and $\gamma_{10,5}=\gamma_{2,1}$.

The  only characters $\psi\in X_{10}^-$ are the characters induced by $\psi_5$ and its complex-conjugate $\OV{\psi_5}$. By (\ref{2.7}), the factor $1-\chi\psi_5(2)=1-\chi(2)i$ occurs in
formula (\ref{2.6}) as well as its complex-conjugate.
We obtain, for $k\in\{1,3\}$,
\BD
  \gamma_{10,k}=-\frac{\chi(10)}2\RE(\OV{\psi_5}(kp)(1-\chi(2)i)\BCHPSF).
\ED
Moreover, we have, for $k\in\{1,2\}$,
\BD
  \gamma_{5,k}=-\frac{\chi(2)}2\RE(\OV{\psi_5}(kp)\BCHPSF).
\ED
Finally, $\gamma_{2,1}=0$.

\MN
{\em Example.} In the case $p=157$, we have $\BCHPSF=-8+2i$, $\chi(10)=1$, $\chi(2)=-1$, $\psi_5(p)=i$ and, thus
\BD
  \delta_{10,0}=\gamma_{10,1}=-\frac 12\RE(-i(1+i)(-8+2i))=3
\ED
and,
\BD
  \gamma_{5,1}=\frac 12\RE(-i(-8+2i))=1.
\ED
Accordingly, $\delta_{10,1}=\gamma_{5,1}-\gamma_{10,1}=1-3=-2$. Formula (\ref{1.7}) gives
\BD
  n_{10,0}=\frac 12\left(\LF{157}{10}+3\right)=9\MB{ and }n_{10,1}=\frac 12\left(\left(\LF{2\cdot 157}{10}-\LF{157}{10}\right)-2\right)=7.
\ED

The case $b=24$ is the most complex of the class number cases. Here we have to deal with characters mod $d$ for certain divisors $d$ of $24$.
We denote them by $\psi_d^+$ and $\psi_d^-$, where the sign $\pm$ is the sign of the respective character.
We write $\WH{\psi_d^+}$ and $\WH{\psi_d^-}$ for the character mod $24$ induced by such a character mod $d$. It turns out that
\BE
\label{3.0}
  X_{24}^+=\{\WH{\psi_1^+}, \WH{\psi_8^+}, \WH{\psi_{12}^+}, \psi_{24}^+\} \MB{ and } X_{24}^-=\{\WH{\psi_3^-}, \WH{\psi_4^-}, \WH{\psi_{8}^-}, \psi_{24}^-\}.
\EE
Here $\psi_1$ is the principal character, $\psi_8^+$ is defined by $\psi_8^+(3)=-1$, $\psi_3^-$ defined by $\psi_3^-(2)=-1$, $\psi_4^-$ defined by $\psi_4^-(3)=-1$, $\psi_8^-$ defined by
$\psi_8^-(3)=1$. Moreover, $\psi_{12}^+=\psi_3^-\psi_4^-$, $\psi_{24}^+=\psi_3^-\psi_8^-$, and $\psi_{24}^-=\psi_3^-\psi_8^+$.

First let $p\equiv 3$ mod $4$ and $(k,24)=1$. Formula (\ref{2.4}) for $\gamma_{b,k}$ involves the Bernoulli numbers
$B_{\chi\WH{\psi_1^+}}=(1-\chi(2))(1-\chi(3))\BCH$, $B_{\chi\WH{\psi_8^+}}=(1+\chi(3))B_{\chi\psi_8^+}$, $B_{\chi\WH{\psi_{12}^+}}=B_{\chi\psi_{12}^+}$, and
$B_{\chi\psi_{24}^+}$. Let $h(m)$ denote the class number of the field $\Q(\sqrt m))$. Then $\BCH=-h(-p)$ (recall $p>3$), $B_{\chi\psi_8^+}=-h(-2p)$, $B_{\chi\psi_{12}^+}=-h(-3p)$,
and $B_{\chi\psi_{24}^+}=-h(-6p)$. Hence (\ref{2.4}) takes the form
\begin{eqnarray}
\label{3.2}
 \gamma_{24,k}&=&h(-p)-\frac{\chi(6)}8((1-\chi(2))(1-\chi(3))h(-p)+\psi_8^+(kp)(1+\chi(3))h(-2p)+\nonumber\\
              &&\psi_{12}^+(kp)h(-3p)+\psi_{24}^+(kp)h(-6p))
\end{eqnarray}

\MN
{\em Example.} Let $p=71$ and $k=1$. We have $h(-p)=7$, $h(-2p)=4$, $h(-3p)=8$, and $h(-6p)=24$. Moreover, $\chi(2)=\chi(3)=1$, $\psi_8^+(p)=1$, $\psi_{12}^+(p)=1$, $\psi_{24}^+(p)=1$.
Therefore, formula (\ref{3.2}) yields
\BD
 \delta_{24,0}=\gamma_{24,1}=7-\frac 18(0+2\cdot 4+1\cdot 8+1\cdot 24)=2.
\ED
We see that each of the four class numbers plays a role in this computation.

The identity $\delta_{24,1}=\gamma_{12,1}-\gamma_{24,1}$ involves the number $\gamma_{12,1}$. In order to compute it, we have to work with
$X_{12}^+=\{\WH{\psi_1^+},\psi_{12}^+\}$, and, therefore, only with the Bernoulli numbers $\BCH=-h(-p)$ and $B_{\chi\psi_{12}^+}=-h(-3p)$.
In the case $p=71$, we obtain $\gamma_{12,1}=5$ and $\delta_{24,1}=5-2=3$. By (\ref{1.7}), $n_{24,0}=2$ and $n_{24, 1}=3$.

Let $p\equiv 1$ mod $4$ and $(k,24)=1$. According to (\ref{3.0}), the Bernoulli numbers entering formula (\ref{2.6}) are $B_{\chi\WH{\psi_3^-}}=(1+\chi(2))B_{\chi\psi_3^-}$,
$B_{\chi\WH{\psi_4^-}}=(1+\chi(3))B_{\chi\psi_4^-}$, $B_{\chi\WH{\psi_8^-}}=(1-\chi(3))B_{\chi\psi_8^-}$, and $B_{\chi\psi_{24}^-}$.
We have $B_{\chi\psi_3^-}=-h(-3p)$, $B_{\chi\psi_4^-}=-h(-p)$, $B_{\chi\psi_8^-}=-h(-2p)$, and $B_{\chi\psi_{24}^-}=-h(-6p)$. Hence formula (\ref{2.6}) can be written
\begin{eqnarray}
\label{3.4}
 \gamma_{24,k}&=&\frac{\chi(6)}8(\psi_3^-(kp)(1+\chi(2))h(-3p)+\psi_4^-(kp)(1+\chi(3))h(-p)+\nonumber\\
              &&\psi_{8}^-(kp)(1-\chi(3))h(-2p)+\psi_{24}^-(kp)h(-6p)).
\end{eqnarray}
Since one of the factors $1+\chi(3)$, $1-\chi(3)$ vanishes, we see that at most three class numbers actually occur in this formula for a specific $p$.

\MN
{\em Example.} Let $p=149$ and $k=1$. We have $h(-3p)=14$, $h(-p)=14$, $h(-2p)=6$, and $h(-6p)=28$. Moreover, $\chi(2)=\chi(3)=-1$, $\psi_3^-(p)=-1$, $\psi_{4}^-(p)=1$, $\psi_8^-(p)=-1$, $\psi_{24}^-(p)=1$.
Therefore, formula (\ref{3.4}) yields
\BD
 \delta_{24,0}=\gamma_{24,1}=\frac 18(0+0+(-1)\cdot 2\cdot 6+1\cdot 28)=2
\ED
and, by (\ref{1.7}), $n_{24,0}=4$.

%%%%%%%%%%%%%%%%%%%%%%%%%%%%%%%%%%%%%%%%%%%%%%%%%%%%%%%%%%%%%%%%%%%%%%%%%%%%%%%%%%%%%%%%%%%%
\section*{5. Reconstruction of Bernoulli numbers $\BCHPS$ from frequencies $n_{b,k}$}
%%%%%%%%%%%%%%%%%%%%%%%%%%%%%%%%%%%%%%%%%%%%%%%%%%%%%%%%%%%%%%%%%%%%%%%%%%%%%%%%%%%%%%%%%%%%

Let $p>3$, $b\ge 2$, $p\NDIV b$, be such that  the (multiplicative) order of $b$ mod $p$ is $(p-1)/2$. Suppose that the frequencies $n_{b,k}$, $k=0,\LD,b-1$, are given. From formula (\ref{1.7})
we obtain the respective numbers $\delta_{b,k}$. Since
\BD
  \gamma_{b,k+1}=\delta_{b,0}+\cdots +\delta_{b,k},
\ED
we know the numbers $\gamma_{b,1}$, \LD, $\gamma_{b,b}$ $(=0)$. Formula (\ref{1.12}) involves the Bernoulli number$ \BCH$, which is $0$, if $p\equiv 1$ mod $4$. If $p\equiv 3$ mod $4$, $\BCH$  can be found by
\BD
 \sum_{k=0}^{b-1}kn_{b,k}=\frac 14(b-1)(p-1)+\frac 12 (b-1)\BCH;
\ED
see \cite[Satz 11]{Gi1}. Because of (\ref{1.10}), we have $\gamma_{b,k}=\gamma_{d,l}$ for a certain divisor $d$ of $b$ and an integer $l$, $1\le l\le d$, $(d,l)=1$.
From (\ref{1.12}), we obtain a system of equations
\BE
\label{4.2}
  (\gamma_{d,l}+\BCH)\frac{\phi(d)}{\chi(d)}=\sum_{\psi\in X_d}\OV{\psi}(-lp)\BCHPS,\enspace l=1,\LD,d,\: (l,d)=1,
\EE
for each $d$. We form the matrix
$\Psi=(\OV{\psi}(l))$, where $l=1,\LD d$, $(l,d)=1$, and $\psi$ runs through $X_d$. This matrix is, essentially, unitary. Indeed, the well-known orthogonality relations for characters show
$\Psi\cdot\OV{\Psi}^T=\phi(d)\cdot I$, $I$ denoting the identity matrix. In other words, if we multiply the left-hand side of (\ref{4.2}) by the matrix $\OV{\Psi}^T/\phi(d)$,
we obtain the numbers $\OV{\psi}(-p)\BCHPS$, $\psi\in X_d$, and, thus, the numbers $\BCHPS$.

In other words, the information encoded in the digit frequencies $n_{b,k}$, $k=0,\LD,b-1$, is equivalent to the information encoded in the Bernoulli numbers $\BCHPS$ for $\psi\in X_d$, $d\DIV b$.

%\bigskip
%\centerline{\bf Competing interests and data availability}

%\MN
%The author declares that there are no competing interests. The paper has no associated data.

%%%%%%%%%%%%%%%%%%%%%%%%%%%%%%%%%%%%%%%%%%%%%%%%%%%%%
%%%%%%%%%%%%%%%%%%%%%%%%%%%%%%%%%%%%%%%%%%%%%%%%%%%%%%%%%%%%%%%%%%%%%%%%%%

\bigskip
\noindent
Institut f\"ur Mathematik \\
Universit\"at Innsbruck   \\
Technikerstr. 13/7        \\
A-6020 Innsbruck, Austria \\
Kurt.Girstmair@uibk.ac.at
% ORCID: 0000-0003-3105-5111}

\end{document}